\documentclass[11pt,xcolor=dvipsnames,svgnames,table,reqno]{amsart}

\usepackage[utf8]{inputenc}
\usepackage{graphicx}
\graphicspath{ {./images/} }
\usepackage{hyperref}
\usepackage{amssymb} 
\usepackage{comment}
\usepackage{amsthm} 
\usepackage{mathtools}
\usepackage{nccmath}
\usepackage{physics} 
\usepackage{tipa} 
\usepackage{tikz-cd} 
\usepackage{yfonts}
\usepackage{enumerate} 
\usepackage[T1]{fontenc} 
\usepackage[inline]{enumitem} 
\usepackage{xspace} 
\usepackage{commath}
\usepackage{tikz-cd}
\usepackage{color} 
\usepackage{pgfplots}
\pgfplotsset{compat=1.18}

\usepackage[left=1.5in, right=1.5in, bottom=1.25in,
height=8in]{geometry}
\usepackage{wrapfig}

\def    \C      {{\mathbb C}}
\def    \R      {{\mathbb R}}
\def \Z {{\mathbb Z}}

\renewcommand{\epsilon}{\varepsilon}

\newtheorem{theorem}{Theorem}[section]

\newtheorem{lemma}[theorem]{Lemma}
\newtheorem{prop}[theorem]{Proposition}
\newtheorem{defi}[theorem]{Definition}

\theoremstyle{remark}
\newtheorem{remark}{Remark}
\AtEndDocument{\bigskip{\footnotesize
 \textsc{Rafael Fernandes, UC Santa Cruz, Santa Cruz, California.} \par  
  \textit{E-mail}:  \texttt{rfernan9@ucsc.edu}\\

 \textsc{Brayan Ferreira, Universidade Federal do Espírito Santo, Vitória, Brazil.} \par  
  \textit{E-mail}:  \texttt{brayan.ferreira@ufes.br}
  }}

\title{Max-min energy of pseudoholomorphic curves and periodic Reeb flows in dimension $3$}
\date{}
\author{Rafael Fernandes and Brayan Ferreira}

\begin{document}

\thanks{The first author was partially supported by the NSF grant DMS-2304206}

\maketitle

\begin{abstract}
In this paper, we make use of elementary spectral invariants given by the max-min energy of pseudoholomorphic curves, recently defined by Michael Hutchings, to study periodic $3$-dimensional Reeb flows. We prove that Zoll contact forms on $S^3$ are characterized by $c_1 = c_2 = \mathcal{A}_{\min}$. This follows from the spectral gap closing bound property and a computation of ECH spectral invariants for Zoll contact forms defined on Lens spaces $L(p,1)$ for $p\geq 1$. The former characterization fails for Lens spaces $L(p,1)$ with $p>1$. Nevertheless, we characterize Zoll contact forms on $L(p,1)$ in terms of ECH spectral invariants. Lastly, we note a characterization of Besse contact forms also holds for elementary spectral invariants analogously to the one obtained by Dan Cristofaro-Gardiner and Mazzucchelli.
\end{abstract}

\tableofcontents

\section{Introduction}

In this paper, we show how elementary spectral invariants can be used to characterize when a contact form is Besse, and when a contact form on the sphere $S^3$ is Zoll. Let $Y$ be a closed $3$-manifold. A $1$-form $\lambda$ on $Y$ is a \textit{contact form} if $\lambda\wedge d\lambda$ is nowhere vanish. For every contact form $\lambda$, we associate the \textit{contact structure} $\xi = ker(\lambda)$, and the \textit{Reeb vector field} $R$, uniquely defined by $d\lambda(R,\cdot) = 0$ and $\lambda(R)=1$. A \textit{Reeb orbit} is a periodic orbit of the Reeb vector field modulo reparametrization of the domain, i.e., is a map $\gamma: \mathbb{R}/ T\mathbb{Z} \rightarrow Y$ such that $\dot{\gamma}(t) = R(\gamma(y))$, and $\gamma$ and $\eta$ Reeb orbits are the same if $\gamma(t) = \eta(t + \theta)$, for some $\theta \in \mathbb{R}$. A Reeb orbit $\gamma$ is called \textit{simple} if $\gamma$ is an embedding, otherwise, $\gamma$ is an $d$-fold of a simple orbit for some integer $d >1$. It was conjectured by Weinstein that any closed contact manifold admits a Reeb orbit. This conjecture is now a theorem in dimension $3$, proved by Taubes in \cite{taubes2007seiberg}. In light of this result, we consider the minimal period (also called the action) of a Reeb orbit on a closed contact $3$-manifold, denoted by $\mathcal{A}_{min}(Y,\lambda)$ or $\mathcal{A}_{min}(\lambda)$.

A contact form is called \textit{Besse} if all trajectories of the Reeb flow are closed Reeb orbits. In this case, all Reeb orbits admit a common period, by a theorem of Wadsley~\cite{wadsley1975geodesic}. When the contact form is Besse and all Reeb orbits share a common minimal period, the contact form is called \textit{Zoll}. There are several works proving that Zoll contact forms are, modulo strictly contactomorphisms, uniquely given by the standard models up to scaling; see \cite{abbondandolo2017systolic,abbondandolo2018sharp,benedetti2021local,chaidez2025zoll}. In~\cite{cristofaro2020action}, Cristofaro-Gardiner and Mazzucchelli proved that a Besse contact form is determined, up to strict contactomorphisms, by its \emph{prime action spectrum}, i.e., the set of periods of simple Reeb orbits. Moreover, they gave a characterization of Besse closed contact $3$-manifolds in terms of the \textit{embedded contact homology} (ECH) spectrum. In particular, they showed that the Besse condition is equivalent to the existence of some $\sigma \in \text{ECH}(Y)$ such that $U\sigma \neq 0$ and $c^{ECH}_{\sigma}(Y,\lambda) = c^{ECH}_{U\sigma}(Y,\lambda)$. 
In the higher-dimensional setting of this family of questions, Ginzburg, Gürel, and Mazzucchelli proved that Besse and Zoll convex contact spheres in $\mathbb{R}^{2n}$ can be characterized using a sequence of spectral invariants of the Clarke action functional \cite{ginzburg2021spectral}. In the realm of differential geometry, many results have been established in the quest to understand the Zoll condition of the geodesic flow on a sphere; see e.g. \cite{mazzucchelli2018characterization} and the references therein. More recently, Ambrozio, Marques and Neves defined a higher-dimensional notion of Zoll Riemannian metrics in the sense of minimal hypersurfaces \cite{ambrozio2025riemannian} and also gave a characterization of such Zoll metrics on $S^3$ in terms of spherical area widths \cite{ambrozio2024rigidity}.

In what follows, we make use of \textit{elementary spectral invariants}, defined by Hutchings in~\cite{hutchings2022elementary}, to characterize again Besse contact $3$-manifolds (similarly to the characterization by Cristofaro-Gardiner and Mazzucchelli), as well as Zoll contact $3$-spheres. In particular, we characterize Zoll star-shaped hypersurfaces in $\mathbb{C}^2$.

\subsection{Main results} Let $(Y,\lambda)$ a contact closed $3$-manifold. For each $k\geq 0$, the contact invariant $c_k(Y,\lambda)$ is defined as the supremum over $R>0$ of the max-min of upper energy of suitable pseudoholomorphic curves on the symplectic completion of $[-R,0]\times Y$, passing through $k$ distinct marked points.

Given a star-shaped hypersurface $\Sigma \subset \C^2$, the standard Liouville form $\lambda_0 = \frac{1}{2}\left(\sum_{i=1}^2 x_i dy_i - y_i dx_i\right)$ restricts to a contact form on $\Sigma$. We say that $\Sigma$ is a \emph{Zoll hypersurface} if $\lambda_0$ restricts to a Zoll contact form on $\Sigma$. Now we are ready to state our first result:




\begin{theorem}\label{thm:zoll iff c1=c2}
A contact form $\lambda$ on $S^3$ is Zoll if and only if $c_1(S^3,\lambda) = c_2(S^3,\lambda) = \mathcal{A}_{\min}(S^3,\lambda)$. In particular, a star-shaped hypersurface $\Sigma \subset \C^2$  is a Zoll hypersurface if and only if $c_1(\Sigma,\lambda_0) = c_2(\Sigma,\lambda_0) = \mathcal{A}_{\min}(\Sigma,\lambda_0)$.
\end{theorem}

We note that one can drop the $\mathcal{A}_{min}(\lambda)$ condition in the Theorem above; see Remark \ref{amincond}.

In \cite{abbondandolo2018sharp}, Abbondandolo, Bramham, Hryniewicz and Salomão proved that for any Zoll contact form $\lambda$ defined in $S^3$ there exists a strictly contactomorphism between $(S^3,C\lambda_0)$ and $(S^3,\lambda)$ for some constant $C >0$, i.e., a diffeomorphism $f\colon S^3 \to S^3$ such that $f^*\lambda = C\lambda_0$. We do not use this fact in the proof of Theorem \ref{thm:zoll iff c1=c2}. Moreover, combining this fact with Theorem \ref{thm:zoll iff c1=c2}, one concludes that there exists such a strictly contactomorphism if and only if $c_1(S^3,\lambda) = c_2(S^3,\lambda) = \mathcal{A}_{\min}(S^3,\lambda)$.

The ``only if'' part of Theorem of \ref{thm:zoll iff c1=c2} follows from the Spectral Gap Closing Bound property of $c_k$, see \eqref{specgapprop}, and Lemma \ref{closel=0} below. The latter was suggested in \cite[Remark 1.4]{hutchings2022elementary}. This part holds for any closed contact $3$-manifold. The ``if'' direction relies on a computation of some ECH spectral invariants for Zoll contact forms on Lens spaces. This computation has its own importance, so we highlight it in our next theorem.

Let $p,q$ be two positive coprime integers. We recall that the Lens space $L(p,q)$ is the $3$-dimensional smooth manifold defined as the quotient space of the $\Z_p$ action
$$(k \mod p) \cdot (z_1,z_2) = (e^{2k\pi i /p}z_1, e^{2k\pi i q/p}z_2)$$
defined on the unit sphere $S^3 \subset \C^2$. Since the standard Liouville form $\lambda_0$ is invariant through this action, $\lambda_0|_{S^3}$ descends to a contact form on $L(p,q)$. Moreover, when $q=1$, the latter contact form is Zoll. We have the following characterization of Zoll contact forms in $L(p,1)$.

\begin{theorem}\label{echlens}
Let $\lambda$ be a contact form on a Lens space $L(p,1)$. Then, $\lambda$ is a Zoll contact form if and only if there exists $\sigma \in ECH(L(p,1),\ker \lambda)$ such that $c^{ECH}_\sigma(L(p,1),\lambda) = c^{ECH}_{U\sigma}(L(p,1),\lambda) = \mathcal{A}_{min}(L(p,1),\lambda)$.
\end{theorem}


We reinforce that the ``only if'' direction in Theorem \ref{echlens} also holds for a general closed contact $3$-manifold $Y$, see Theorem \ref{closelech} and Lemma \ref{closel=0} below.

\begin{remark}
    We recall that given a Finsler metric $F$ on a closed surface $\Sigma$, the unit tangent bundle $S_F\Sigma$ admits the Hilbert contact form $\alpha$ such that the Reeb orbits are of the form $(c,\dot{c})$, where $c \colon I \to \Sigma$ is a closed geodesic on $(\Sigma,F)$ parametrized by arc length. Moreover, the action of $(c,\dot{c})$ coincides with the length $L(c) = \int_I F(c,\dot{c})$, see e.g. \cite{dorner2017finsler}. In particular, Theorem \ref{echlens} gives a characterization of Zoll Finsler metrics on $S^2$ or $\R P^2$ in terms of ECH spectral invariants since $S_FS^2$ and $S_F\R P^2$ are diffeomorphic to $L(2,1)$ and $L(4,1)$, respectively.
\end{remark}

\begin{remark}\label{amincond}
    In \cite[Theorem 1.1]{mazzucchelli2023structure}, Mazzucchelli and Radeschi proved that every Besse $(S^3,\lambda)$ is strictly contactomorphic to an ellipsoid $(E(a,b),\lambda_0)$ with $a/b \in \mathbb{Q}$. Using this fact and computing the invariants in the latter case, one can drop the $\mathcal{A}_{min}(\lambda)$ condition in Theorem \ref{thm:zoll iff c1=c2}, i.e., it is automatically satisfied. From an analogous result for Lens spaces, one can also drop the $\mathcal{A}_{min}(\lambda)$ condition in Theorem \ref{echlens} in the cost of asking for the existence of an ECH class $\sigma$ such that $U\sigma \neq 0$, $c^{ECH}_\sigma(L(p,1),\lambda) = c^{ECH}_{U\sigma}(L(p,1),\lambda)$ but $c^{ECH}_{U^2\sigma}(L(p,1),\lambda) = 0$.
\end{remark}

Since the “only if” direction in Theorem \ref{thm:zoll iff c1=c2} holds in full generality, a natural question is how large the class of closed contact manifolds is for which Theorem \ref{thm:zoll iff c1=c2} holds.


Our next result shows that \( Y = S^3 = L(1,1) \) is the unique Lens space \( L(p,1) \) for which the characterization in Theorem~\ref{thm:zoll iff c1=c2} holds.

\begin{theorem}\label{c2lp1}
    Let $p>1$ and $\lambda_0$ be the standard Zoll contact form on $L(p,1)$ induced by the standard Liouville form on $\C^2$. Then $c_1(L(p,1),\lambda_0) = \mathcal{A}_{min}(\lambda_0)$ and $c_2(L(p,1),\lambda_0) = 2\mathcal{A}_{min}(\lambda_0)$.
\end{theorem}

The contact invariants $c_k(Y,\lambda)$ are way simpler to define in comparison with ECH spectral invariants $c^{ECH}_\sigma(Y,\lambda)$. However, Theorem \ref{echlens} and Theorem \ref{c2lp1} show that the ECH spectral invariants are capable of realizing more dynamical properties than the max-min invariants $c_k(Y,\lambda)$ in some sense.

Now turning to the characterization of Besse contact forms, we prove the following result inspired by \cite[Lemma 3.1]{cristofaro2020action}, and give a similar characterization for Besse contact forms in terms of the max-min invariants $c_k$.

\begin{theorem}\label{besse}
Let $\lambda$ be a contact form on a closed three-manifold $Y$. Then $\lambda$ is a Besse contact form if and only if $c_k(Y,\lambda) = c_{k+1}(Y,\lambda)$ for some $k > 0$.
\end{theorem}

The arguments for the proof follows closely those of Cristofaro-Gardiner and Mazzucchelli, as it depends only on properties of the elementary spectral invariants which are entirely analogous to those of the ECH spectral invariants. For the sake of completeness, we provide the proof below.
\vskip 8pt
\noindent{\bf Structure of the paper:}
In Section 2, we recall the definition of ECH spectral invariants and review some of their properties. In Section 3, we introduce the max–min invariants $c_k(Y,\lambda)$, summarize their properties and prove Theorem \ref{besse}. Section 4 is devoted to Zoll contact forms on closed three-manifolds: we relate meromorphic sections of suitable complex line bundles to pseudoholomorphic curves in the symplectization of a Zoll contact manifold, compute certain ECH spectral invariants for $L(p,1)$, and prove Theorem \ref{echlens}. Finally, in Section 5, we prove Theorem \ref{thm:zoll iff c1=c2} and Theorem \ref{c2lp1}.

\vskip 8pt

\noindent{\bf Acknowledgments:}
The second author would like to thank Lucas Ambrozio for many helpful and inspiring conversations on the problem of characterizing Zoll manifolds. We thank Viktor Ginzburg for asking a question that motivated Remark \ref{amincond}.

\section{ECH spectral invariants and closing property}

One of the key ingredients that goes into the main theorem is the use of ECH spectral invariants, which we now recall. Let $\lambda$ be a nondegenerate contact form on $Y$. For a generic almost complex structure $J$ that is admissible in $\R \times Y$, the ECH chain complex $ECC(Y,\lambda,J)$ is a $\Z_2$ vector space generated by suitable (finite) orbit sets $\alpha = \{(\alpha_i,m_i)\}$. Here, $\alpha_i$ are embedded Reeb orbits on $(Y,\lambda)$ and $m_i$ are positive integers. The differential is defined by counting $J$-holomorphic currents asymptotic to orbit sets with \emph{ECH index} $1$. This index is a contact topological number associated to relative homology classes $Z \in H_2(Y,\alpha,\beta)$, where $\alpha$ and $\beta$ are chain complex generators. 

After a highly non-trivial analysis, it turns out that the differential is well defined and $\partial^2 = \partial \circ \partial = 0$. In particular, we have a well-defined homology $ECH(Y,\lambda,J)$, see \cite{hutchings2014lecture} for a good introduction on this topic.

\begin{theorem}[Taubes \cite{taubes2010embedded}]\label{taubes}
There is a canonical isomorphism of relatively graded modules
\[ECH_*(Y,\lambda,\Gamma,J) = \widehat{HM}^{-*}(Y,\mathfrak{s}_\xi + PD(\Gamma)).\]
\end{theorem}
In the theorem above, $\widehat{HM}^{-*}(Y,\mathfrak{s}_\xi + PD(\Gamma))$ denotes the ``from" version of Seiberg--Witten Floer cohomology defined by Kronheimer and Mrowka in \cite{kronheimer2007monopoles}.

In particular, the homology $ECH(Y,\lambda,J)$ does not depend on $\lambda$ or on $J$, and then we write $ECH(Y,\xi)$.

Since the differential decreases total action of orbit sets, there is also a well-defined filtered homology $ECH^L(Y,\lambda,J)$ given $L>0$ generated by orbit sets whose total action $\mathcal{A}(\alpha) = \sum_i \int_{\alpha_i} \lambda$ is less than $L$.
In \cite{hutchings2013proof}, Hutchings and Taubes proved this filtered homology does not depend on $J$ and that the map induced by inclusion
$$\imath_L \colon ECH^L(Y,\lambda) \to ECH(Y,\xi)$$
also does not depend on $J$. Given a nonzero class $\sigma \in ECH(Y,\xi)$, we define its \emph{ECH spectral invariant} as the number
$$c_\sigma^{ECH}(Y,\lambda) = \inf\{ L \mid \sigma \in \imath_L(ECH^L(Y,\lambda))\}.$$

We extend this definition to degenerate contact forms using the fact that this invariant is $C^0$-continuous with respect to $\lambda$.

There is an additional structure on ECH given by a degree $-2$ map. This map, usually denoted by $U$, is defined in the chain complex level and counts ECH index $2$ $J$-holomorphic currents connecting chain complex generators and passing through $(y,0) \in \R \times Y$, where $y$ is a generic point on $Y$ not contained in a Reeb orbit. It turns out that this is a well-defined homology map that also does not depend on $\lambda$ or $J$. Moreover, it is shown by Taubes in \cite{taubes2010embeddedV} that this map coincides with the analogous $U$-module structure on the Seiberg--Witten Floer cohomology under the isomorphism in Theorem \ref{taubes}.

Now, we recall how the gap in the spectral invariants can be used to study the closedness of orbits.
\begin{defi}
    Let $(Y,\lambda)$ be a closed contact $3$-manifold, and let $U$ be a non empty open set. A \textit{positive deformation} of $\lambda$ supported in $U$ is a smooth one parameter family of contact forms 
    $$\{\lambda_{t} = e^{f_t}\lambda\}_{t\in [0,1]},$$
satisfying:
\begin{itemize}
    \item $f_0 \equiv 0.$
    \item $f_t = 0$ for all $t \in [0,1]$ and $y\in Y\backslash U$.
    \item $f_1 \geq 0$, and $f_1$ does not vanish identically.
\end{itemize}
\end{defi}
Now, for a contact manifold $(Y,\lambda)$, we can associate its symplectization 
$$(\mathbb{R}\times Y, \omega = d(e^s\lambda)),$$
where $s$ denotes the $\mathbb{R}$ parameter. Recall that if $(M,\omega)$ is symplectic $4$-manifold, then its \textit{Gromov width}
$$c_{Gr}(M,\omega)  \in (0,\infty],$$
is defined as the supremum of $a>0$, such that there exists a symplectic embedding 
$$B^4(a) \longrightarrow (M,\omega),$$
where 
$$B^4(a) = \{z \in \mathbb{C}^2 \mid \pi |z|^2 \leq a \},$$
endowed with the restriction of canonical symplectic form from $\mathbb{C}^2$.
\begin{defi}
    For a positive deformation $\{\lambda_t = e^{f_t} \lambda\}_{t\in [0,1]}$, consider
    $$M_{\lambda_1} = \{(s,y) \in \mathbb{R}\times Y \mid 0<s<f_1(y)\},$$
    endowed with the restriction of the symplectic form $d(e^s\lambda)$. We define the \textit{width} of the positive deformation as
    $$\mathrm{Width} (\lambda_t) = c_{Gr}(M_{\lambda_1}) \in (0,\infty).$$
\end{defi}
\begin{defi}
    Let $(Y,\lambda)$ be a close contact $3$-manifold, and let $L>0$. Define 
    $$\mathrm{Close}^L(Y,\lambda) \in [0,\infty],$$
    to be the infimum of $\delta>0$ satisfying the following property:\\ \\
    For any $\{\lambda_t\}$ positive deformation of $\lambda$ with $\mathrm{Width}(\lambda_t) \geq \delta$, if supported in a nonempty  set $U$, then there exists  $t \in [0,1]$, such that the contact form $\lambda_t$ has a Reeb orbit intersecting $U$ with action at most $L$.
\end{defi}

\begin{remark} \label{rmk:CloseL infinity smaller then the minimal action}
    If $L$ is less then the minimum action of a Reeb orbit of $(Y,\lambda)$, then for any $\delta>0$, by taking $U=Y$ and $\lambda_t = e^{rt}\lambda$, with $r>0$ sufficiently large with respect to $\delta$, then we conclude $\mathrm{Close}^L(Y,\lambda)>\delta$, therefore $\mathrm{Close}^L(Y,\lambda) = \infty$.
\end{remark}
The following property of \(\text{Close}^L(Y,\lambda)\) is important for what follows, and we highlight it in the next lemma, providing a proof for the sake of completeness.

\begin{lemma}\label{closel=0}
Let $L>0$. Then $\mathrm{Close}^L(Y,\lambda) = 0$ holds if and only if every point in $Y$ is contained in a Reeb orbit of action at most $L$.
\begin{proof} Assume that $(Y,\lambda)$ is Besse, and let $L$ be larger than the period of any simple Reeb orbit. Then, for any positive deformation $\{\lambda_{\tau}\}_{\tau\in [0,1]}$, the form $\lambda_0 = \lambda$ has a Reeb orbit passing through its support with action at most $L$. Therefore $\mathrm{Close}^L(Y,\lambda) = 0$.

Now, assume $\text{Close}^L(Y,\lambda) = 0$. Take $p \in Y$, and consider a sequence of bump functions $f_n$ supported in neighborhoods $U_n$ of $p$ decreasing to $p$. Then, it is not hard to see that for each $n$, the positive deformation $\{\lambda_{\tau} = \tau  f_n\}_{\tau \in [0,1]}$ has $\text{Width}(\lambda_{\tau}) >0$, see \cite[Proposition 6.2]{edtmair2021pfh}). Then, for some $\tau \in [0,1]$, the contact form $\lambda_{\tau}$ has a Reeb orbit $\gamma_n$ passing through $U_n$. From Arzelà–Ascoli theorem, we can extract a subsequence of $\gamma_n$ converging to a loop $\gamma$. Then, from the choice positive deformation, it follows that $\gamma$ is a Reeb orbit of $\lambda$ with action at most $L$ that passes through $p$.
\end{proof}
\end{lemma}

We recall how the gap in the ECH spectral invariants bounds the number $\mathrm{Close}^L(Y,\lambda)$.
\begin{theorem}[Hutchings \cite{hutchings2022elementary}]\label{closelech}
    $$\mathrm{Close}^L(Y,\lambda) \leq \inf \{c_\sigma^{ECH}(Y,\lambda) - c_{U\sigma}^{ECH}(Y,\lambda) \mid \ U\sigma \neq 0, \ c_\sigma^{ECH} \leq L\}.$$
\end{theorem}

It follows from the definition of the chain complex that the ECH can be decomposed by singular homology classes

$$ECH(Y,\xi) = \bigoplus_{\Gamma \in H_1(Y;\Z)}ECH(Y,\xi,\Gamma).$$

Using the $U$ map structure, Hutchings defined a distinguished sequence of ECH spectral invariants in the case where the contact invariant $c(\xi) = [\emptyset] \in ECH(Y,\xi,0)$ is nonzero. This occurs in the case where $(Y,\xi)$ is symplectically fillable, for instance, see \cite[Example 1.10]{hutchings2014lecture}. This sequence is often called the \emph{ECH spectrum of $(Y,\lambda)$} and it is defined as
$$c_k^{ECH}(Y,\lambda) = \min  \{c_\sigma(Y,\lambda) \mid \ \sigma \in ECH(Y,\xi,0), \ U^k\sigma = [\emptyset]\}$$
for a nondegenerate contact form $\lambda$ on $Y$ and again extended by continuity for the degenerate case.

In the next section we recall the definition of the max-min invariants $c_k(Y,\lambda)$ recently defined by Hutchings and some of their properties.

\section{Max-min invariants from pseudoholomorphic curves}

\subsection{Definition of $c_k$ and its properties}In this section, we briefly recall the definition of the elementary spectral invariants defined by Hutchings and discuss their properties.

Let \((Y,\lambda)\) be a closed contact \(3\)-manifold, and consider its symplectization $\mathbb{R}\times Y$. For each \(R>0\), we denote by \(\mathcal{J}(\overline{[-R,0]\times Y})\) the set of cobordism-compatible almost complex structures on the cobordism \([-R,0]\times Y\); see \cite{hutchings2022elementary} for more details. Notice that \([-R,0]\times Y\) is a strong symplectic cobordism between \((Y,e^{-R}\lambda)\) and \((Y,\lambda)\), and its completion is the symplectization $\mathbb{R}\times Y$. For \(J \in \mathcal{J}(\overline{[-R,0]\times Y})\), and orbit sets \(\alpha_+\) of \((Y,\lambda)\) and \(\alpha_-\) of \((Y,e^{-R}\lambda)\), we denote by
\[
\mathcal{M}^J(\overline{[-R,0]\times Y},\alpha_+,\alpha_-)
\]
the set of pseudoholomorphic curves in the completion of \([-R,0]\times Y\) asymptotic to \(\alpha_{+}\) at their positive punctures and to \(\alpha_-\) at their negative ones; see Section~2 of \cite{hutchings2022elementary}. For \(u \in \mathcal{M}^J(\overline{[-R,0]\times Y},\alpha_+,\alpha_-)\), we denote by \(\mathcal{E}_+(u)\) its upper energy, defined by \(\mathcal{E}_+(u)= \mathcal{A}(\alpha_+)\). We are now ready to recall the definition of the max-min contact invariants $c_k(Y,\lambda)$.

\begin{defi}\label{def:ck}
    Let $(Y,\lambda)$ a closed contact $3$-manifold, and $k\geq0$. When $\lambda$ is nondegenerate, we define
    $$c_k(Y,\lambda) =\sup_{R>0} \sup_{\substack{J \in \mathcal{J}([-R,0]\times Y) \\ x_1,\ldots,x_k \in [-R,0]\times Y  \ \text{distinct}}} \inf_{u \in \mathcal{M}^J(\mathbb{R}\times Y, x_1,\ldots,x_k)}
\mathcal{E}_+(u).$$
When $\lambda$ is degenerate, we define 
$$c_{k}(Y,\lambda) = \inf_{f>0} c_{k}(e^f\lambda) = \sup_{f<0}c_k(Y,e^f\lambda),$$
where the supremum and infimum are taken over the set of smooth functions in $Y$ such that $e^f\lambda$ is nondegenerate.
\end{defi}
Notice that the definition above is not practical for computations for degenerate contact forms. We remark that there is a way to make it in terms of pseudoholomorphic curves.
\begin{remark}\label{rmkbourg} As pointed out in \cite[Lemma 3.8]{bourgeois2002morse}, the pseudoholomorphic curves considered in the definition of \(c_k(Y,\lambda)\) for the nondegenerate case coincide with pseudoholomorphic curves satisfying similar properties (not necessarily asymptotic at punctures) but with finite \textit{Hofer energy}. The latter set of curves converges to orbits at the punctures, as proved in \cite{hofer1993pseudoholomorphic}. Although pseudoholomorphic curves with finite Hofer energy may fail to converge to a single orbit when approaching the punctures in the case of a degenerate contact form, they still approach orbits (possibly different ones), as points approach the punctures via sequences, as demonstrated in \cite{hofer8properties}. An example where the orbits coming from a puncture are different was given in \cite{siefring2017finite}. Nonetheless, all such orbits share a common period (called the \textit{mass}) and we can define \(c_k\) in terms of these curves for the degenerate case. The two definitions coincide, since the definition above is the continuous extension of the nondegenerate case.

\end{remark}
Now we recall some properties of the contact invariants $c_k(Y,\lambda)$.

\begin{theorem}[Hutchings \cite{hutchings2022elementary}]\label{ckproperties}
Let $(Y,\lambda)$ be any closed contact three-manifold. The invariants $c_k$ have the following properties:
\begin{enumerate}

\item (Conformality) If $r > 0$, then
$$c_k(Y, r\lambda) = r c_k(Y,\lambda).$$

\item (Sublinearity) $$c_{k+l}(Y,\lambda) \leq c_k(Y,\lambda) + c_l(Y,\lambda).$$

\item (Spectrality)\label{spectrality} For a given $k$, there exists an orbit set $\alpha$ such that $c_k(Y,\lambda) = \mathcal{A}(\alpha)$.

\item(Liouville Domains) If $(Y,\lambda)$ is the boundary of a Liouville domain $(X,\omega)$, then
$$c_k(Y,\lambda) \leq c_k^{Alt}(X,\omega) \leq c_k^{ECH}(Y,\lambda).$$

\item (Sphere) $$c_k(S^3,\lambda_0) = d,$$
where $d$ is the unique nonnegative integer such that
$$d^2 + d \leq 2k \leq d^2 + 3d.$$

\item (Asymptotics) \label{asymptotic} $$\lim_{k\to \infty} \frac{c_k(Y,\lambda)^2}{k} = 2\mathrm{Vol}(Y,\lambda).$$

\item (Spectral Gap Closing Bound) If $k > 0$ and $c_k(Y,\lambda) \leq L$, then
\begin{equation}\label{specgapprop}
\mathrm{Close}^L(Y,\lambda) \leq c_k(Y,\lambda) - c_{k-1}(Y,\lambda).
\end{equation}
\end{enumerate}
\end{theorem}

Most of these properties are stated and proven in \cite[Theorem 1.14]{hutchings2022elementary}. The Asymptotics property is given in \cite[Theorem 1.19]{hutchings2022elementary} and the Liouville Domains property is a combination of the corresponding property in \cite[Theorem 1.14]{hutchings2022elementary} and \cite[Theorem 12]{hutchings2022elementary2}.

We are now ready to prove Theorem \ref{besse}.
\subsection{Characterization of Besse contact forms}
\begin{proof}[Proof of Theorem \ref{besse}]
    Assume that $c_{k}(Y,\lambda) = c_{k-1}(Y,\lambda)=L_0$ for some integer $k$. Then from the spectral gap property we have $\mathrm{Close}^{L_0}(Y,\lambda)= 0$, and from Lemma \ref{closel=0}, we conclude that $(Y,\lambda)$ is Besse.
    
    Now, we assume that $(Y,\lambda)$ is Besse. From Wadsley's Theorem \cite{wadsley1975geodesic}, there exists $\widetilde{T}$ positive such that for every orbit $\gamma$ of the Reeb flow $\varphi^t_{\lambda}$ has minimal period $\widetilde{T}/k_{\gamma}$, for some $k_{\gamma} \in \mathbb{Z}_{>0} =\{1,2,...\}$. Since $Y$ is compact and the Reeb vector field is nowhere vanishing, there is a positive number $\mathcal{A}_{min}(\lambda)$ that bounds from below the period of any closed orbit. Therefore, since $\Tilde{T}/\mathcal{A}_{min}(\lambda) \geq k_{\gamma}$ for any $\gamma$ Reeb orbit, the set 
    $$K=\{k_{\gamma}  \mid \  \gamma \text{ is a Reeb orbit }\},$$
    is bounded from above. For $m$ the minimal common multiple of the elements in $K$, we can conclude that any Reeb orbit has period a multiple of $T = \Tilde{T}/m$. Therefore, from the spectrality property in \ref{spectrality}, any $c_{k}(Y,\lambda)$ is an element of the set $\{nT \ \mid \ n \in \mathbb{Z}_{> 0}\}$. 

    If we assume that $c_{k+1}(Y,\lambda)>c_k(Y,\lambda)$ for all $k \geq 0$, then $c_{k+1}(Y,\lambda)\geq c_k(Y,\lambda) + T$, and therefore, $c_{k}(Y,\lambda)\geq kT$, for any $k \geq 0$, but this implies 
    $$\lim_{k \rightarrow \infty} c_{k}^2(Y,\lambda)/k \geq \lim_{k\rightarrow \infty} kT^2 =\infty,$$
    which contradicts the asymptotic property of $c_{k}(Y,\lambda)$ in \ref{asymptotic}.
\end{proof}

\section{Zoll contact forms}

\subsection{Prequantization bundles}\label{sec:preqb}
Let $Y$ be a three-dimensional closed manifold, and let $\lambda$ be a contact form defined on $Y$ that is Zoll. For $S^1 = \R / (\mathcal{A}_{\min}(\lambda)\Z)$, we have a natural smooth $S^1$-action
\begin{align}\label{s1action}
S^1 \times Y &\to Y \nonumber \\
(\theta, y) &\mapsto \theta \cdot y := \varphi_\theta(y),
\end{align}
where $\varphi \colon \R \times Y \to Y$ denotes the Reeb flow on $(Y,\lambda)$. Since the action \eqref{s1action} is free, the orbit space $Y/S^1 =: \Sigma$ is a smooth compact surface. In fact, we have a $S^1$-principal bundle (see \cite[Theorem 2]{boothby1958contact}).
$$\begin{tikzcd}
 \frac{\R}{\mathcal{A}_{\min}(\lambda) \Z}=S^1\arrow{r} &Y \arrow{d}{\mathfrak{q}} \\ &\Sigma
\end{tikzcd}$$
Consider the set
$$Y \times_{S^1} \C = \{[y,z] \in Y \times \C \mid (y,z) \sim (\theta \cdot y, e^{i\theta}z) \ \text{for some $\theta$ in $S^1$}\}.$$
Then, the canonical projection $Y \times_{S^1} \C \to \Sigma,\ [y,z] \mapsto \mathfrak{q}(y),$ is a complex line bundle. We denote it by $L = Y \times_{S^1} \C$. One can choose a Hermitian metric on $L$ such that the corresponding unit sphere bundle coincides with $Y \cong Y \times \{0\} \subset L$. In particular, $Y$ is the boundary of a Liouville domain. In this case, we have
$$c_1(L) = c_1(Y) \in H^2(\Sigma;\Z).$$

Since $TY = \langle R \rangle \oplus \ker \lambda$, we have a natural bundle isomorphism $\ker \lambda \cong T\Sigma$. In particular, there exists a unique $\omega \in \Omega^2(\Sigma)$ such that $\mathfrak{q}^*\omega = d\lambda$. It turns out that $\omega$ is a symplectic form, provided that $d\lambda\vert_{\ker \lambda}$ is nondegenerate. Therefore, the orbit space $\Sigma$ is an orientable compact surface. From now on, we denote its genus by $g$. Moreover, the cohomology class $-[\omega]/\mathcal{A}_{\min}(\lambda)$ is integral and coincides with the first Chern class $c_1(L) \in H^2(\Sigma;\Z)$; see, e.g., \cite[\S 2]{kobayashi1956principal} or \cite[Theorem 2]{boothby1958contact}.

Using Gysin sequence or spectral sequence for the bundle $S^1 \hookrightarrow Y \rightarrow \Sigma$ one can compute the homology of the total space:
\begin{equation}
H_*(Y;\Z) = 
\begin{cases}
\Z, \quad *=3, \\
\Z^{2g}, \quad *=2 \\
\Z^{2g} \oplus \Z_{-c_1}, \quad *= 1, \\
\Z, \quad *=0,
\end{cases}
\end{equation} 
where $c_1 = c_1(Y)[\Sigma] \in \Z$ denotes the Chern number of the circle bundle $\mathfrak{q}\colon Y \to \Sigma$. Moreover, given a point $p \in \Sigma$, the fiber $\mathfrak{q}^{-1}(p)$ over $p$ is such that its $k$-fold cover represents the class $k \mod (-c_1)$ in the $\Z_{-c_1}$ summand of $H_1(Y;\Z)$, see e.g. \cite[\S 3.2]{nelson2020embedded}. 

Note also that we have the following diffeomorphism:
\begin{align}\label{symplectization}
F \colon \R \times Y &\to L^* := L \setminus 0_L \nonumber \\ (r,y) &\mapsto [y,e^r],
\end{align}
where $0_L \subset L$ denotes the zero section. Therefore, we can identify the symplectization $(\R \times Y, d(e^r\lambda))$ with $(L^*,(F^{-1})^*(d(e^r\lambda)))$.

\subsection{Meromorphic sections and pseudoholomorphic curves}\label{sec:meromorphic}

Given any almost complex structure $J$ on the symplectization $\R \times Y$, we obtain a corresponding almost complex structure on $L^*$ via the identification $F$ in \eqref{symplectization}. If $J$ is $\R$-invariant, the latter induces an integrable complex structure $j$ on $\Sigma \cong 0_L \subset L$ such that the projection $\pi:=[y,z] \mapsto \mathfrak{q}(y)$ is a holomorphic map. Hence, $L \to \Sigma$ can be seen as a holomorphic line bundle over the compact Riemann surface $(\Sigma,j)$. 

From now on, we focus on special pseudoholomorphic curves whose domain is $\Sigma$ with some punctures, i.e., smooth maps
\begin{equation}\label{specialpseudo}
u \colon \Sigma \setminus \Gamma \to \R \times Y \cong L^*,
\end{equation}
such that $du \circ j = J \circ du$ and $\Gamma = \{p_1^-,\ldots,p_l^-,p_1^+,\ldots,p_m^+\} \subset \Sigma$ is a set of points regarded as punctures.

Note that, through the identification \eqref{symplectization}, positive and negative ends of $u$ correspond precisely to poles and zeros of the composition $F \circ u$, respectively. In particular, a meromorphic section $s \colon \Sigma \to L$ can be interpreted as a pseudoholomorphic curve of the form \eqref{specialpseudo}. Further, since such a meromorphic section is a proper map, $J$-holomorphic curves of this type have finite Hofer energy, see \cite[Lemma 3.8]{bourgeois2002morse}.

We recall that the degree $\deg(L)$ of the holomorphic line bundle $L$ agrees with the Chern number $c_1(L)[\Sigma] \in \Z$. The following two lemmas are based on the theory of meromorphic sections of holomorphic line bundles over compact Riemann surfaces and are useful to estimate $c_k$ in the Zoll case.

\begin{lemma}\label{numberofpoles}
    Let $\pi \colon L \to \Sigma$ be a holomorphic line bundle over a compact Riemann surface $\Sigma$ with degree $\deg(L) = -d, d\geq 1$. Then, any meromorphic section $s\colon \Sigma \to L$ of $\pi \colon L \to \Sigma$ must have (counting multiplicity) at least $d$ poles. In particular, $\pi\colon L \to \Sigma$ does not admit a meromorphic section with total multiplicity of poles $<d$.
    \begin{proof}
        If $s\colon \Sigma \to L$ is a nonzero meromorphic section of $\pi \colon L \to \Sigma$, we have $\mathrm{div}(s) = Z - P$, where $Z$ and $P$ are the effective divisors of zeros and poles of $s$, respectively. In this case, we have
        $$-d = \deg(L) = \deg(\mathrm{div}(s)) = \deg(Z) - \deg(P).$$
        Therefore, $\deg(P) = \deg(Z) + d \geq d.$
    \end{proof}
\end{lemma}

\begin{lemma}\label{existsasection}
Let $\pi \colon L \to S^2$ be a holomorphic line bundle over the sphere with degree $\deg(L) = -d, d\geq 1$ and $v_1,v_2 \in L$ be two points in different fibers. Then, for $k> d + 1$, there exists a meromorphic section $s\colon S^2 \to L$ with total multiplicity of poles given by $k$ and such that $s(x_i) = v_i, \ i=1,2$. 
\begin{proof}
From Riemann--Roch Theorem, we have
    $$h^0(E) - h^1(E) = \deg(E) -g + 1,$$
    for any holomorphic line bundle $E \to \Sigma_g$, where $g$ is the genus of the compact Riemann surface $\Sigma_g$, $h^0(E) = \dim H^0(S^2,E) $ and $h^1 = \dim H^1(S^2,E)$. Let $P$ be an effective divisor with degree $k>d + 1$. Note that a meromorphic section of $L$ with poles in $P$ is equivalent to a holomorphic section of the bundle $L \otimes \mathcal{O}(P) \cong \mathcal{O}(k-d)$. In this case, we have $g = 0$ and 
    $$h^1(L \otimes \mathcal{O}(P)) = h^0(S^2, K \otimes (L \otimes \mathcal{O}(P))^{-1}) = h^0(S^2, \mathcal{O}(-2+d-k)) = 0$$
    by Serre duality for $k-d+2 \geq 0$. Therefore,
    \begin{align}\label{dimensionh0}
        h^0(L \otimes \mathcal{O}(P)) &= \deg(L \otimes \mathcal{O}(P)) + 1 \nonumber \\ &= k - d + 1
    \end{align}
    Now, identifying the space $H^0(S^2,L \otimes \mathcal{O}(P))$ of holomorphic sections of the bundle $L \otimes \mathcal{O}(P)$ with the space of meromorphic sections of $L$ with poles in $P$, we consider the evaluation map
    \begin{align*}
        ev_{x_1,x_2} \colon H^0(S^2,L \otimes \mathcal{O}(P)) &\to L_{x_1} \oplus L_{x_2}
        \\ s &\mapsto (s(x_1)-v_1,s(x_2)-v_2).
    \end{align*}
Note that by \eqref{dimensionh0}, the kernel of this map has dimension given by
\begin{equation}
h^0(S^2,L \otimes \mathcal{O}(P)) - (\dim L_{x_1} + \dim L_{x_2})= k - d -1.
\end{equation}
Therefore, if $k>d + 1$, there must exist a meromorphic section $s$ of $L$ with poles in $P$ such that $s(x_i) = v_i$, $i =1,2$.
\end{proof}
\end{lemma}
Given a $J$-holomorphic curve $u \colon \Sigma\setminus \Gamma \to \R \times Y$, under the identification $F$ in \eqref{symplectization}, we obtain a meromorphic map $\overline{u} \colon \Sigma \to L$ such that the positive punctures in $\Gamma$ are poles of $\overline{u}$ and negative punctures in $\Gamma$ are zeros of $\overline{u}$. The following Lemma gives a criterion to ensure whether $u$ comes from a meromorphic section of the complex line bundle $\pi \colon L \to \Sigma$.

\begin{lemma}[cf. Exercise 4.5 in \cite{hutchings2014lecture}]\label{oncesection}
    Let $u \colon \Sigma\setminus \Gamma \to \R \times Y$ be a $J$-holomorphic curve, where $J$ is a $\R$-invariant almost complex structure on $\R \times Y$. Suppose that the image of the map $F \circ u \colon \Sigma\setminus \Gamma \to L^* = L \setminus 0_L$ intersects each fiber of the holomorphic line bundle $\pi \colon L^* \to \Sigma$ once and transversely except from the fibers over the asymptotic Reeb orbits. Then $u$ comes from a meromorphic section $s_u \colon \Sigma \to L$.
    \begin{proof}
        Let $u\colon \Sigma \setminus \Gamma \to \R \times Y$ be such a $J$-holomorphic curve. Then, for $x \in \Sigma$ not coming from an asymptotic Reeb orbit, we have $\# (F(u(\Sigma \setminus \Gamma) \cap L_x) = 1$, where $L_x = \pi^{-1}(x)$ is the fiber over $x$. Say $F(u(\Sigma \setminus \Gamma) \cap L_x = \{v_x\}$. It is clear that $s_u \colon \Sigma \to L$ defined by $s_u(x) = v_x$, for $x \in \Sigma \setminus \Gamma$, and extending naturally to its zeros and poles following the asymptotics of $u$, is a section of the line bundle $\pi\colon L \to \Sigma$. We claim that $s_u$ is a meromorphic section. To prove this, we note that $f := \pi \circ F \circ u \colon \Sigma \setminus \Gamma \to \Sigma$ is a holomorphic map. Since $F \circ u$ intersects each fiber of $\pi \colon L^* \to \Sigma$ once and transversely, the holomorphic map $f$ is injective and hence is a biholomorphism over its image. In this case, the section $s_u$ coincides with the holomorphic map $F \circ u \circ f^{-1}$ outside its zeros and poles, and hence, $s_u$ is a meromorphic section.
        
$$\begin{tikzcd}
\Sigma \setminus \Gamma \arrow[r, "F \circ u"] \arrow[rd, "f"'] & L \arrow[d] \\
& \Sigma \arrow[u, bend right, dotted, "s_u"']
\end{tikzcd}$$
    \end{proof}
\end{lemma}

\subsection{ECH spectral invariants for Zoll $L( p ,1)$}
Let $\lambda$ be a contact form on the Lens space $L( p,1)=: Y$. In this case, we must have $g=0$ and the circle bundle $\mathfrak{q} \colon Y \to S^2$ is such that $c_1 = - p$. We follow a Bourgeois approach introduced in \cite{bourgeois2002morse}, see also \cite{nelson2020embedded} and \cite{ferreira2022symplectic}. Given a perfect Morse function $f\colon S^2 \to \R$, we consider the perturbed contact form
$$\lambda_\varepsilon:= (1+\varepsilon \mathfrak{q}^*f)\lambda.$$
Note that $\ker \lambda_\varepsilon = \ker \lambda = \xi$ for every $\varepsilon>0$. Recalling the splitting $TY = \langle R \rangle \oplus \xi$, where $R$ is the Reeb vector field, one can understand the periodic orbits of the new Reeb vector field for a fixed action range.

\begin{lemma}[Lemma 2.3 in \cite{bourgeois2002morse}]\label{critpts}
    For each $T>0$, there exists $\varepsilon = \varepsilon(T)>0$ such that the periodic orbits of $R_{\lambda_\varepsilon}$ in $Y$ with action lower than or equal to $T$ are nondegenerate and correspond to the critical points of $f$.
\end{lemma}

For a critical point $p$ of $f$, we denote the corresponding Reeb orbit and its iterations by $\gamma_p^k$, $k = 1,2,\ldots$. Bourgeois computed the Conley--Zehnder indices of these orbits for $\varepsilon>0$ as in the previous Lemma.

\begin{lemma}[Lemma 2.4 in \cite{bourgeois2002morse}]\label{lemmacz}
    Consider $T>0,\ \varepsilon>0$ as in the previous Lemma. Suppose that $k$ is a positive integer such that $k\mathcal{A_{\lambda_\varepsilon}}(\gamma_p)\leq T$. Then the Conley--Zehnder index of $\gamma_p^k$ as a Reeb orbit of $\lambda_\varepsilon$ is given by
    $$CZ_\tau(\gamma_p^k) = \mu^\tau(\gamma_p^k)-1+\mathrm{ind}_f(p),$$
    where $\tau$ is a trivialization $\tau$ of $\xi$ along $\gamma_p$, $\mu^\tau(\gamma_p^k)$ is the Robbin--Salamon index of $\gamma_p^k$ as a Reeb orbit of $\lambda$ and $\mathrm{ind}_f(p)$ denotes the Morse index of the function $f$ at the critical point $p$.
\end{lemma}

Since $f$ is a perfect Morse function on $S^2$, we have only two critical points say $p_1$, $p_2$ with $\mathrm{ind}_f(p_1) = 0$ and $\mathrm{ind}_f(p_2)=2$. In this case, Lemma \ref{lemmacz} yields
$$CZ_\tau(\gamma_{p_1}^k) = \mu^\tau(\gamma_{p_1}^k) - 1 \quad \text{and} \quad  CZ_\tau(\gamma_{p_2}^k) = \mu^\tau(\gamma_{p_2}^k) + 1.$$

Given $p \in S^2$ and $q \in \mathfrak{q}^{-1}(p) \subset Y$, the natural identification $\xi_q \cong T_pS^2$ allow us to obtain a trivialization of $\xi_q$ out of a trivialization of $T_pS^2$. This is often called as the \emph{constant trivizalition} over the Reeb orbit $\gamma_p = \mathfrak{q}^{-1}(p)$ since the linearized Reeb flow with respect to this trivialization is given by the identity map. Denoting by $\tau$ such a constant trivialization of $\xi$ along a circle fiber $\gamma_p = \mathfrak{q}^{-1}(p)$, one computes $\mu^\tau(\gamma^k_p) = 0$ for every $k$, see \cite[Lemma 3.8]{nelson2020embedded}. In particular, for this trivialization, we get
$$CZ_\tau(\gamma_{p_1}^k) = - 1 \quad \text{and} \quad  CZ_\tau(\gamma_{p_2}^k) = 1.$$
Therefore, since these are odd numbers for all iterates, $\gamma_{p_1}$ and $\gamma_{p_2}$ are elliptic\footnote{i.e., the eigenvalues of the linearized Reeb flow restricted to $\xi$ lies on the unit circle.} regarded as Reeb orbits for the contact form $\lambda_\varepsilon$. In this case, the ECH index parity property in \cite[\S 3.4]{hutchings2014lecture} ensures that the ECH differential
\begin{equation}\label{diffvanish}
    \partial \colon ECC_*^T(L(p,1),\lambda_\varepsilon,J) \to ECC_*^T(L(p,1),\lambda_\varepsilon,J)
\end{equation}
vanishes for any almost complex structure $J$. As an immediate consequence, we have the following computation.

\begin{prop}\label{filtedredech}
    For each $T>0$ and $\Gamma \in H_1(L(p,1);\Z) \cong \Z_p$, there exists $\varepsilon(T)>0$ such that the filtered embedded contact homology group $ECH_*^T(L(p,1),\lambda_{\varepsilon(T)}, \Gamma)$ is the $\Z_2$ vector space generated by orbit sets $\{(\gamma_{p_1},m_1), (\gamma_{p_2},m_2)\}$ such that $m_1 + m_2 \equiv \Gamma \mod p$ and
    $$m_1 \mathcal{A}_{\lambda_\varepsilon}(\gamma_{p_1}) + m_2 \mathcal{A}_{\lambda_\varepsilon}(\gamma_{p_2}) < T.$$
    \begin{proof}
    As we recalled in Section \ref{sec:preqb}, we have $H_1(L(p,1;\Z) \cong \Z_p$ and the homology of the fiber $\gamma_p = \mathfrak{q}^{-1}(p)$ over $p \in S^2$ is such that its $k$-fold cover represents the class $k \mod p$. In particular, the total homology $m_1[\gamma_{p_1}] + m_2[\gamma_{p_2}]$ is given by $m_1+m_2 \mod p$. The fact that $ECH_*^T(L(p,1),\lambda_\varepsilon, \Gamma)$ is generated by such orbit sets now follows from Lemma \ref{critpts}, the fact that the ECH differential in \eqref{diffvanish} vanishes and the definition of filtered ECH.
\end{proof}
\end{prop}

Using a filtered version of the Taubes isomorphism in Theorem \ref{taubes} obtained in \cite{hutchings2013proof}, Morgan and Weiler obtained the following result for any prequantization bundle $\mathfrak{q} \colon Y \to \Sigma_g$.

\begin{theorem}[Proposition 3.2 + Theorem 7.1 in \cite{nelson2020embedded}]\label{directlimit}
    For any $\Gamma \in H_1(Y;\Z)$ the filtered ECH groups $ECH_*^T(Y,\lambda_{\varepsilon(T)},\Gamma)$ forms a direct system. Moreover, the direct limit recovers the original ECH group
    $$\lim_{T \to \infty} ECH_*^T(Y,\lambda_{\varepsilon(T)},\Gamma) = ECH_*(Y,\xi,\Gamma),$$
    where $\xi = \ker \lambda_\varepsilon$ is the contact structure defined by $\lambda_\varepsilon$.
\end{theorem}

Now we recall the ECH groups for Lens spaces $L(p,1)$. For this, see \cite[Example 1.3]{nelson2020embedded} or \cite[Corollary 3.4]{kronheimer2007monopoleslens} combined with Taubes isomorphism in Theorem \ref{taubes}.

\begin{theorem}\label{echlensspaces}
    The $\Z$-graded module ECH for the Lens space $L(p,1)$ is given by
    $$ECH_*(L(p,1),\xi,\Gamma) = \begin{cases}
        \Z_2, \quad * \in 2\Z, \\ 0, \quad \text{else},
    \end{cases}$$
    for each singular homology class $\Gamma \in H_1(L(p,1);\Z) \cong \Z_p$. Moreover, if $\zeta_k$ is the generator of $ECH_{2k}(L(p,1),\xi,\Gamma)$, then $U\zeta_k = \zeta_{k-1}$ for $k\geq 1$.
\end{theorem}

We are now ready to prove that $\mathcal{A}_{min}(\lambda)$ is an ECH spectral value for a Zoll contact form $\lambda$ on $L(p,1)$.

\begin{lemma}\label{echspeclp1}
Let $\lambda$ be a Zoll contact form on $L(p,1)$. Then, there exists a nonzero class $\sigma \in ECH(L(p,1),\xi)$ such that $c^{ECH}_\sigma(L(p,1),\lambda) = c^{ECH}_{U\sigma}(L(p,1),\lambda) = \mathcal{A}_{min}(\lambda)$.
\begin{proof}
    The ECH class induced by the representative $\{(\gamma_{p_2},1)\}$ in the direct limit given in Theorem \ref{directlimit} satisfies this property. In fact, for each $T\gg \mathcal{A}_{min}(\lambda)$, we have 
    $$\mathcal{A}_{\lambda_\varepsilon}(\gamma_{p_1}) = (1+\varepsilon f(p_1))\mathcal{A}_{min}(\lambda) < (1+\varepsilon f(p_2))\mathcal{A}_{min}(\lambda) = \mathcal{A}_{\lambda_\varepsilon}(\gamma_{p_2}) < 2\mathcal{A}_{\lambda_\varepsilon}(\gamma_{p_1}).$$
    Since $U\zeta_k = \zeta_{k-1}$ and the $U$ map decreases the action, for $\Gamma \equiv 1 \mod p$, we conclude that $\zeta_0 , \zeta_1 \in ECH_2(L(p,1),\xi,\Gamma)$ must be the generators induced by the representatives $\{(\gamma_{p_1},1)\}$ and $\{(\gamma_{p_2},1)\}$, respectively. Therefore, we compute
    $$c^{ECH}_{\zeta_1}(L(p,1),\lambda) = \lim_{T \to \infty} \mathcal{A}_{\lambda_\varepsilon}(\gamma_{p_2}) = \lim_{\varepsilon \to 0} (1+\varepsilon f(p_2))\mathcal{A}_{min}(\lambda) = \mathcal{A}_{min}(\lambda).$$
    Similarly, we have
    $$c^{ECH}_{\zeta_0}(L(p,1),\lambda) = \lim_{T \to \infty} \mathcal{A}_{\lambda_\varepsilon}(\gamma_{p_1}) = \lim_{\varepsilon \to 0} (1+\varepsilon f(p_1))\mathcal{A}_{min}(\lambda) = \mathcal{A}_{min}(\lambda).$$
    Since $U\zeta_1 = \zeta_0$, this proves this Lemma. We note that in the sphere case, $S^3 = L(1,1)$, the proof goes word-by-word changing $\zeta_1, \zeta_0$ by $\zeta_2$ and $\zeta_1$, respectively, since $\zeta_0 = [\emptyset] \in ECH_0(S^3,\xi,0)$.
\end{proof}
\end{lemma}

\begin{remark}
    A similar argument proves that the ECH spectral invariants for Zoll contact forms on Lens spaces $L(p,1)$ are given by
\begin{align*}
\mathcal{S}^{ECH}(L(p,1),\lambda_{Zoll}) &= \{c^{ECH}_{\sigma}(L(p,1),\lambda_{Zoll})\mid \sigma \in ECH_*(L(p,1),\ker \lambda_{Zoll})\} \\ &= \{n\mathcal{A}_{min}\mid n \in \Z_{\geq 0}\}.
\end{align*}
Moreover, for each $k \mod p$, the ECH classes in $ECH_*(L(p,1),\ker \lambda_{Zoll},\Gamma)$, for $\Gamma \equiv k \mod p$, are precisely the classes with ECH spectral invariants given by $(m_1+m_2)\mathcal{A}_{min}$ such that $m_1,m_2 \in \Z_{\geq 0}$ and $m_1 + m_2 \equiv k \mod p$.
\end{remark}

\begin{proof}[Proof of Theorem \ref{echlens}]
If a class $\sigma \in ECH(L(p,1),\xi)$ is such that
$$c^{ECH}_\sigma(L(p,1),\lambda) = c^{ECH}_{U\sigma}(L(p,1),\lambda) = \mathcal{A}_{min}(\lambda)>0,$$
then $\sigma$ and $U\sigma$ are nonzero in $ECH(L(p,1),\xi)$. In this case, Theorem \ref{closelech} yields $$\mathrm{Close}^{\mathcal{A}_{min}(\lambda)}(L(p,1),\lambda) = 0.$$ By Lemma \ref{closel=0}, this implies that every point in $L(p,1)$ is contained in a Reeb orbit of action at most $\mathcal{A}_{min}(\lambda)$, and hence, $\lambda$ is a Zoll contact form. This part holds for any closed contact $3$-manifold $Y$.

On the other hand, suppose that $\lambda$ is a Zoll contact form on $L(p,1)$. By Lemma \ref{echspeclp1}, there exists a nonzero class $\sigma \in ECH(L(p,1),\xi)$ such that $c^{ECH}_\sigma(L(p,1),\lambda) = c^{ECH}_{U\sigma}(L(p,1),\lambda) = \mathcal{A}_{min}(\lambda)$ and this is our desired conclusion.
\end{proof}

\section{Proof of Theorems \ref{thm:zoll iff c1=c2} and \ref{c2lp1}}\label{proofthm}

\begin{proof}[Proof of Theorem \ref{thm:zoll iff c1=c2}]
    As we noted in the Introduction, the ``only if'' part of Theorem of \ref{thm:zoll iff c1=c2} follows from the Spectral Gap Closing Bound in \eqref{specgapprop} and Lemma \ref{closel=0} and holds for any closed $3$-manifold. It suffices then to prove that $c_2(S^3,\lambda) = c_1(S^3,\lambda) = \mathcal{A}_{min}(\lambda)$ for any Zoll contact form $\lambda$ on $S^3$.
    
    Let $\lambda$ be a Zoll contact form on $S^3 = L(1,1)$. From the spectrality property it follows
    \begin{equation}\label{lowerbound}
        \mathcal{A}_{min}(\lambda) \leq c_1^{ECH}(S^3,\lambda) \leq c_2^{ECH}(S^3,\lambda)
    \end{equation}
Moreover, by the Liouville domain property, we know
\begin{equation}\label{liouvibound}
    c_k(S^3,\lambda) \leq c_k^{ECH}(S^3,\lambda).
\end{equation}
Recall that by Theorem \ref{echlensspaces} we have
$$ECH(S^3,\xi,0) = \bigoplus_{k \in \Z_{\geq 0}} \zeta_{k}\Z_2,$$
and $U\zeta_k = \zeta_{k-1}$ for $k\geq 1$. From the proof of Lemma \ref{echspeclp1}, we have $c^{ECH}_{\zeta_2}(S^3,\lambda) = \mathcal{A}_{min}(\lambda)$. Since $U^2\zeta_2 = U \zeta_1 = \zeta_0 = [\emptyset]$, we get
\begin{equation}\label{upperbound}
    c_2^{ECH}(S^3,\lambda) \leq c^{ECH}_{\zeta_2}(S^3,\lambda) = \mathcal{A}_{min}(\lambda).
\end{equation}
The result now follows combining the spectrality property of $c_k$, \eqref{lowerbound}, \eqref{liouvibound} and \eqref{upperbound}.
\end{proof}

Finally, we prove Theorem \ref{c2lp1}

\begin{proof}[Proof of Theorem \ref{c2lp1}]
    We start noting that from Theorem \ref{echlens}, there exists a class $\sigma \in ECH(L(p,1),\xi)$ such that
    $$c^{ECH}_\sigma(L(p,1),\lambda_0) = c^{ECH}_{U\sigma}(L(p,1),\lambda_0) = \mathcal{A}_{min}(\lambda_0)>0.$$
    In particular, $U\sigma \neq 0 \in ECH(L(p,1),\xi)$. In \cite[Theorem 6.1]{hutchings2022elementary}, Hutchings proved that
    $$c_k(Y,\lambda) \leq \inf \{c_\sigma^{ECH}(Y,\lambda) \mid U^k\sigma \neq 0\}.$$
    Therefore, we have $c_1(L(p,1),\lambda_0) \leq c_\sigma^{ECH}(L(p,1),\lambda_0) = \mathcal{A}_{min}(\lambda_0)$. Thus, by spectrality property, $c_1(L(p,1),\lambda_0) = \mathcal{A}_{min}(\lambda_0)$ must hold. Now it follows from the sublinearity property that
    $$c_2(L(p,1),\lambda_0) \leq c_1(L(p,1),\lambda_0) + c_1(L(p,1),\lambda_0) = 2\mathcal{A}_{min}(\lambda_0).$$
    It suffices to prove that $c_2(L(p,1),\lambda_0) \geq 2\mathcal{A}_{min}(\lambda_0)$ and again by spectrality property, it is enough to check $c_2(L(p,1),\lambda_0) > \mathcal{A}_{min}(\lambda_0)$ since $\lambda_0$ is a Zoll contact form. For this, we first note that given $R>0$, we have
    \begin{equation}\label{c2lowbound}
        c_2(L(p,1),\lambda_0) \geq \sup_{\substack{J \in \mathcal{J}^\R(\lambda) \\ v_1,v_2 \in [-R,0] \times L(p,1)}} \inf_{u \in \mathcal{M}^J(\R \times L(p,1); v_1,v_2)}\mathcal{E}_+(u),
        \end{equation}
    where $\mathcal{J}^\R(\lambda)$ denotes the set of symplectization admissible almost complex structures in $\R \times L(p,1)$. In particular, $J \in \mathcal{J}^\R(\lambda)$ is $\R$-invariant. The inequality \eqref{c2lowbound} holds since the completion $\overline{[-R,0] \times L(p,1)}$ coincides with the symplectization $\R \times L(p,1)$ and in the definition of $c_2$ we consider more general almost complex structures rather than just the symplectization admissible ones, see Definition \ref{def:ck} and Remark \ref{rmkbourg}.
    
    Let $J \in \mathcal{J}^\R(\lambda)$. Note that by Lemma \ref{existsasection}, there exists a meromorphic section $s \colon S^2 \to L$ with multiplicity of poles given by $k>p+1 > 2$ and such that $s(x_i) = v_1$ for $i=1,2$, whenever $v_1,v_2 \in L$ are two points in different fibers. Therefore, there exists a $J$-holomorphic curve $u \colon S^2 \setminus \Gamma \to \R \times L(p,1)$ with positive ends covering Reeb orbits of $\lambda_0$ with total multiplicity $k> 2$ passing through $v_1,v_2 \in [-R,0] \times L(p,1)$ for $v_1 = (s_1,y_1)$ and $v_2 = (s_2,y_2)$, where $y_1,y_2 \in L(p,1)$ are points lying in different Reeb orbits.
    
    To conclude that $c_2(L(p,1),\lambda_0) > \mathcal{A}_{min}(\lambda_0)$ it is enough to prove that there exist an almost complex structure  $J \in \mathcal{J}^\R(\lambda)$ and points $v_1,v_2 \in [-R,0] \times L(p,1)$ such that there is no $J$-holomorphic curve $u \colon \Sigma \setminus \{x\} \to \R \times L(p,1)$ with a unique positive end converging to a simple Reeb orbit of $\lambda_0$. We claim that this holds for any $J \in \mathcal{J}^\R(\lambda)$.\\
    \textbf{Case $g>0$:} Suppose that there exists such a $J$-holomorphic curve with $\Sigma = \Sigma_g$ being a compact Riemann surface with positive genus $g$. Following the discussion in Section \ref{sec:meromorphic}, we have the induced meromorphic map $\overline{u} \colon \Sigma_g \to L$ with exactly one simple pole and no zeros. In this case, the composition $f=\pi \circ \overline{u} \colon \Sigma \to S^2$ is a holomorphic map.
    $$\begin{tikzcd}
\Sigma \arrow{r}{\overline{u}} \arrow{rd}{f} &L \arrow{d}{\pi} \\ &S^2
\end{tikzcd}$$
    As such, $f$ must be constant or surjective with degree $\deg(f)>1$ since $g>0$. If $f$ is constant, $u$ would be a trivial cylinder ``$\R \times \gamma$'' and this contradicts $g>0$. Therefore, $f$ has degree at least $2$. In particular, we have the following diagram
     $$\begin{tikzcd}
f^*L \arrow{d} &L \arrow{d}{\pi} \\ \Sigma \arrow{r}{f} &S^2,
\end{tikzcd}$$
where $f^*L = \{(x,y) \in \Sigma \times L \mid f(x) = \pi(y)\} \to \Sigma$ given by $(x,y) \mapsto x,$ is a holomorphic line bundle with degree $\deg(f) \deg(\pi)$. Since $\pi \colon L \to S^2$ is a holomorphic line bundle with degree $-p$, we have $\deg(f^*L) = -p \deg(f) \leq -2p$. Lastly, if $u$ had just one positive end converging to a simple Reeb orbit, the map $x \mapsto (x,\overline{u}(x))$ would be a meromorphic section of $f^*L \to \Sigma$ with exactly one simple pole. Since $\deg(f^*L) \leq -2p$ and $p>1$, this contradicts Lemma \ref{numberofpoles}. \\
\textbf{Case $g=0$:} Suppose that there exists a $J$-holomorphic curve $u \colon S^2 \setminus \{x\} \to \R \times L(p,1)$ with a unique positive end converging to a simple Reeb orbit $\gamma$ of $\lambda_0$. Recall that the Reeb orbits of $\lambda_0$ are projections of the Hopf fibers and, hence, given two simple Reeb orbits $\gamma_1$ and $\gamma_2$ in $(L(p,1),\lambda_0)$, we have $sl(\gamma_i^p) = -p, i=1,2,$ and $lk(\gamma_1^p, \gamma_2^p) = p$. Given a degree $p$ holomorphic map $\phi \colon S^2 \to S^2$, the composition $u \circ \phi \colon S^2 \setminus \phi^{-1}(x) \to \R \times L(p,1)$ is a $J$-holomorphic curve with positive punctures converging with total multiplicity to a $p$-cover of $\gamma$. 

By the definition of linking numbers, the curve $u \circ \phi$ intersects each fiber of the holomorphic line bundle $\pi \colon L^* \to S^2$ $p$ times transversely except from the fibers over the asymptotic Reeb orbit $\gamma$. In this case, the curve $J$-holomorphic curve $u \colon S^2\setminus \{x\} \to \R \times L(p,1)$ intersects each such a fiber once and transversely and thus, by Lemma \ref{oncesection}, $u$ comes from a meromorphic section $s_u \colon S^2 \to L$ with exactly a simple pole. This contradicts Lemma \ref{numberofpoles} since we have $\deg(L) = -p$ with $p>1$.
\end{proof}

\bibliographystyle{alpha}
\bibliography{mybibliography}

\begin{thebibliography}{KMOS07}

\bibitem[ABHS17]{abbondandolo2017systolic}
Alberto Abbondandolo, Barney Bramham, Umberto~L Hryniewicz, and Pedro~AS
  Salom{\~a}o.
\newblock A systolic inequality for geodesic flows on the two-sphere.
\newblock {\em Mathematische Annalen}, 367(1):701--753, 2017.

\bibitem[ABHS18]{abbondandolo2018sharp}
Alberto Abbondandolo, Barney Bramham, Umberto~L Hryniewicz, and Pedro~AS
  Salom{\~a}o.
\newblock {Sharp systolic inequalities for Reeb flows on the three-sphere}.
\newblock {\em Inventiones mathematicae}, 211(2):687--778, 2018.

\bibitem[AMN24]{ambrozio2024rigidity}
Lucas Ambrozio, Fernando~C Marques, and Andr{\'e} Neves.
\newblock {Rigidity theorems for the area widths of Riemannian manifolds}.
\newblock {\em arXiv preprint arXiv:2408.14375}, 2024.

\bibitem[AMN25]{ambrozio2025riemannian}
Lucas Ambrozio, Fernando~C Marques, and Andr{\'e} Neves.
\newblock {Riemannian metrics on the sphere with Zoll families of minimal
  hypersurfaces}.
\newblock {\em Journal of Differential Geometry}, 130(2):269--341, 2025.

\bibitem[BK21]{benedetti2021local}
Gabriele Benedetti and Jungsoo Kang.
\newblock A local contact systolic inequality in dimension three.
\newblock {\em J. Eur. Math. Soc.(JEMS)}, 23(3):721--764, 2021.

\bibitem[Bou02]{bourgeois2002morse}
Fr{\'e}d{\'e}ric Bourgeois.
\newblock {\em {A Morse-Bott approach to contact homology}}.
\newblock PhD thesis, Stanford university, 2002.

\bibitem[BW58]{boothby1958contact}
William~M Boothby and Hsieu-Chung Wang.
\newblock On contact manifolds.
\newblock {\em Annals of Mathematics}, 68(3):721--734, 1958.

\bibitem[CGM20]{cristofaro2020action}
Daniel Cristofaro-Gardiner and Marco Mazzucchelli.
\newblock {The action spectrum characterizes closed contact 3-manifolds all of
  whose Reeb orbits are closed}.
\newblock {\em Commentarii Mathematici Helvetici}, 95(3):461--481, 2020.

\bibitem[CT25]{chaidez2025zoll}
Julian Chaidez and Shira Tanny.
\newblock {On Zoll Contact $5$-Spheres}.
\newblock {\em arXiv preprint arXiv:2509.09639}, 2025.

\bibitem[DGZ17]{dorner2017finsler}
Max D{\"o}rner, Hansj{\"o}rg Geiges, and Kai Zehmisch.
\newblock {Finsler geodesics, periodic Reeb orbits, and open books}.
\newblock {\em European Journal of Mathematics}, 3(4):1058--1075, 2017.

\bibitem[EH21]{edtmair2021pfh}
Oliver Edtmair and Michael Hutchings.
\newblock {PFH spectral invariants and $C^\infty$ closing lemmas}.
\newblock {\em arXiv preprint arXiv:2110.02463}, 2021.

\bibitem[FR22]{ferreira2022symplectic}
Brayan Ferreira and Vinicius~GB Ramos.
\newblock Symplectic embeddings into disk cotangent bundles.
\newblock {\em Journal of Fixed Point Theory and Applications}, 24(3):62, 2022.

\bibitem[GGM21]{ginzburg2021spectral}
Viktor~L Ginzburg, Ba{\c{s}}ak~Z G{\"u}rel, and Marco Mazzucchelli.
\newblock {On the spectral characterization of Besse and Zoll Reeb flows}.
\newblock {\em Annales de l'Institut Henri Poincar{\'e} C}, 38(3):549--576,
  2021.

\bibitem[Hof93]{hofer1993pseudoholomorphic}
Helmut Hofer.
\newblock {Pseudoholomorphic curves in symplectizations with applications to
  the Weinstein conjecture in dimension three}.
\newblock {\em Inventiones mathematicae}, 114(1):515--563, 1993.

\bibitem[HT13]{hutchings2013proof}
Michael Hutchings and Clifford Taubes.
\newblock {Proof of the Arnold chord conjecture in three dimensions, II}.
\newblock {\em Geometry \& Topology}, 17(5):2601--2688, 2013.

\bibitem[Hut14]{hutchings2014lecture}
Michael Hutchings.
\newblock Lecture notes on embedded contact homology.
\newblock In {\em Contact and symplectic topology}, pages 389--484. Springer,
  2014.

\bibitem[Hut22a]{hutchings2022elementary2}
Michael Hutchings.
\newblock An elementary alternative to {ECH} capacities.
\newblock {\em Proceedings of the National Academy of Sciences},
  119(35):e2203090119, 2022.

\bibitem[Hut22b]{hutchings2022elementary}
Michael Hutchings.
\newblock Elementary spectral invariants and quantitative closing lemmas for
  contact three-manifolds.
\newblock {\em arXiv preprint arXiv:2208.01767}, 2022.

\bibitem[HWZ]{hofer8properties}
H~Hofer, K~Wysocki, and E~Zehnder.
\newblock {Properties of pseudoholomorphic curves in symplectisation. IV.
  Asymptotics with degeneracies. Contact and symplectic geometry (Cambridge,
  1994), 78-117}.
\newblock {\em Publ. Newton Inst}, 8.

\bibitem[KM07]{kronheimer2007monopoles}
Peter~B Kronheimer and Tomasz Mrowka.
\newblock {\em Monopoles and three-manifolds}, volume~10.
\newblock Cambridge University Press Cambridge, 2007.

\bibitem[KMOS07]{kronheimer2007monopoleslens}
Peter Kronheimer, Tomasz Mrowka, Peter Ozsv{\'a}th, and Zolt{\'a}n Szab{\'o}.
\newblock Monopoles and lens space surgeries.
\newblock {\em Annals of mathematics}, pages 457--546, 2007.

\bibitem[Kob56]{kobayashi1956principal}
Shoshichi Kobayashi.
\newblock Principal fibre bundles with the $1$-dimensional toroidal group.
\newblock {\em Tohoku Mathematical Journal, Second Series}, 8(1):29--45, 1956.

\bibitem[MR23]{mazzucchelli2023structure}
Marco Mazzucchelli and Marco Radeschi.
\newblock {On the structure of Besse convex contact spheres}.
\newblock {\em Transactions of the American Mathematical Society},
  376(03):2125--2153, 2023.

\bibitem[MS18]{mazzucchelli2018characterization}
Marco Mazzucchelli and Stefan Suhr.
\newblock {A characterization of Zoll Riemannian metrics on the $2$-sphere}.
\newblock {\em Bulletin of the London Mathematical Society}, 50(6):997--1006,
  2018.

\bibitem[NW20]{nelson2020embedded}
Jo~Nelson and Morgan Weiler.
\newblock Embedded contact homology of prequantization bundles.
\newblock {\em arXiv preprint arXiv:2007.13883}, 2020.

\bibitem[Sie17]{siefring2017finite}
Richard Siefring.
\newblock Finite-energy pseudoholomorphic planes with multiple asymptotic
  limits.
\newblock {\em Mathematische Annalen}, 368(1):367--390, 2017.

\bibitem[Tau07]{taubes2007seiberg}
Clifford~Henry Taubes.
\newblock {The Seiberg--Witten equations and the Weinstein conjecture}.
\newblock {\em Geometry \& Topology}, 11(4):2117--2202, 2007.

\bibitem[Tau10a]{taubes2010embedded}
Clifford~Henry Taubes.
\newblock Embedded contact homology and {S}eiberg--{W}itten {F}loer cohomology
  {I}.
\newblock {\em Geometry \& Topology}, 14(5):2497--2581, 2010.

\bibitem[Tau10b]{taubes2010embeddedV}
Clifford~Henry Taubes.
\newblock Embedded contact homology and {S}eiberg--{W}itten {F}loer cohomology
  {V}.
\newblock {\em Geometry \& Topology}, 14(5):2961--3000, 2010.

\bibitem[Wad75]{wadsley1975geodesic}
Andrew~W Wadsley.
\newblock Geodesic foliations by circles.
\newblock {\em Journal of Differential Geometry}, 10(4):541--549, 1975.

\end{thebibliography}
\end{document}